\documentclass[11pt,amscd]{amsart}
\usepackage{amssymb}

\usepackage[colorlinks,bookmarks=false,linkcolor=blue,linktocpage]{hyperref}

\newtheorem{Theorem}{Theorem}[section]
\newtheorem{Lemma}[Theorem]{Lemma}
\newtheorem{Corollary}[Theorem]{Corollary}
\newtheorem{Proposition}[Theorem]{Proposition}
\newtheorem{Remark}[Theorem]{Remark}
\newtheorem*{rem*}{Remark}

\newtheorem{Example}[Theorem]{Example}

\def\ker{\operatorname{ker}}
\def\lex{{\operatorname{lex}}}

\def\Spec{\operatorname{Spec}}
\def\depth{\operatorname{depth}}
\def\hom{\operatorname{hom}}
\def\reg{\operatorname{reg}}
\def\Tor{\operatorname{Tor}}

\def\height{\operatorname{ht}}

\def\NP{{\operatorname{N\PP}}}

\def\mm{{\mathfrak m}}

\def\RR{{\mathbb R}}
\def\QQ{{\mathbb Q}}
\def\PP{{\mathbb P}}
\def\ZZ{{\mathbb Z}}
\def\NN{{\mathbb N}}
\def\w{{\bold w}}

\def\a{{\alpha}}

\newcommand{\x}{X}
\newcommand{\Kx}{k[\x]}

\begin{document}

\title{Toward a theory of monomial preorders}

\author{Gregor Kemper}
\address{Technische Universi\"at M\"unchen, Zentrum Mathematik - M11,
Boltzmannstr. 3, 85748 Garching, Germany}
\email{kemper@ma.tum.de}

\author{Ngo Viet Trung}
\address{Institute of Mathematics, Vietnam Academy of Science and Technology, 18 Hoang Quoc Viet, 10307 Hanoi, Vietnam}
\email{nvtrung@math.ac.vn}
 
\author{Nguyen Thi Van Anh}
\address{University of Osnabr\"uck, Institut f\"ur Mathematik, Albrechtstr. 28 A, 49076 Osnabr\"uck, Germany}
\email{ngthvanh@gmail.com}

\thanks{The second author is supported by Vietnam National Foundation for Science and Technology Development under grant number 101.04-2017.19.
A large part of the paper was completed during a long term visit of the second author to Vietnam Institute of Advanced Study on Mathematics.} 


\begin{abstract}
In this paper we develop a theory of monomial preorders, which differ from the classical notion of monomial orders in that they allow ties between monomials. 
Since for monomial preorders, the leading ideal is less degenerate than for monomial orders, 
our results can be used to study problems where monomial orders fail to give a solution.
Some of our results are new even in the classical case of monomial
  orders and in the special case in which the leading ideal defines the
  tangent cone. 
\end{abstract}
 
\maketitle

\section*{Introduction}

A monomial order or a monomial ordering is a total order on the monomials of a polynomial ring which is compatible with the product operation \cite{GP}.  Gr\"obner basis theory is based on monomial orders with the additional condition that 1 is less than all other monomials.  Using such a monomial order, one can associate to every ideal a leading ideal that has a simple structure and that can be used to get information on the given  ideal. 
This concept has been extended to an arbitrary monomial order in order to deal with the local case by Mora, Greuel and Pfister \cite{GP',GP,Mo}. One may ask whether there is a similar theory for partial orders on the monomials of a polynomial ring.

For a partial order, the leading ideal is no longer a monomial ideal
and, therefore, harder to study.  On the other hand, it is closer to
the given ideal in the sense that it is less degenerate than the
leading ideal for a monomial order. An instance is the initial ideal
generated by the homogeneous components of lowest degree of the
polynomials of the given ideal, which corresponds to the notion of the
tangent cone at the origin of an affine variety. 
Being closer to the original ideal, 
a partial order may help to solve a problem
that cannot be solved by any monomial order. 
A concrete example is Cavaglia's proof~\cite{Cav} of a conjecture of Sturmfels on the Koszul property of the pinched Veronese.  
The aim of this paper is to establish an effective theory of partial monomial orders and to show that it has potential applications in the study of polynomial ideals. \par  

Let $\Kx = k[x_1,...,x_n]$ be a polynomial ring over a field $k$. For any integral vector $a = (\a_1,...,\a_n) \in \NN^n$ we write $x^a$ for the monomial $x_1^{\a_1}\cdots x_n^{\a_n}$. Let $<$ be an arbitrary partial order on the monomials of $\Kx$. For every polynomial $f = \sum c_ax^a$ one defines the {\em leading part} of $f$ as 
$$L_<(f) := \sum_{x^a \in \max_<(f)} c_ax^a,$$
where $\max_<(f)$ denotes the set all monomials $x^a$ of $f$ such that there is no monomial $x^b$ of $f$ with $x^a <x^b$. \par

The first problem that we have to address is for which partial orders  the leading parts of polynomials behave well under the operations of $\Kx$. Obviously, such a partial order should be a weak order, i.e. it satisfies the additional condition that incomparability is an equivalence relation. Moreover,  it should be  compatible and cancellative with the product operation, i.e. if $x^a, x^b$ are monomials with $x^a < x^b$, then $x^ax^c < x^bx^c$ for any monomial $x^c$, and if $x^ax^c < x^bx^c$ for some $x^c$, then $x^a < x^b$. If a partial order $<$ satisfies these conditions, we call it a {\em monomial preorder}.  A natural instance  is the {\em weight order} associated to a weight vector $w \in \RR^n$, defined by $x^a < x^b$ if $w \cdot a < w\cdot b$.  \par

We shall see that a binary relation $<$ on the monomials of $\Kx$ is a monomial preorder if and only if there exists a real $m \times n$ matrix $M$ for some $m \ge 1$ such that $x^a < x^b$ if and only if $M \cdot a <_\lex M \cdot b$ for any monomials $x^a, x^b$, where $<_\lex$ denotes the lexicographic order. This means  that monomial preorders are precisely products of weight orders. This characterization is a natural extension of a result of Robbiano \cite{Ro}, who showed that every monomial order can be defined as above by a real matrix with additional properties. It can be also deduced from a subsequent result of Ewald and Ishida in \cite{EI}, where similar preorders on the lattice $\ZZ^n$ were studied from the viewpoint of algebraic geometry (see also Gonzalez Perez and Teissier \cite{GT}). They call the set of all such preorders  the Zariski-Riemann space of the lattice, and use this result to prove the quasi-compactness of that  space. \par

As one can see from the above characterization by real matrices, monomial preorders give rise to graded structures on $\Kx$. For graded structures, Robbiano \cite{Ro2} developed a framework for dealing with leading ideals. See also the papers of Mora~\cite{Mo3} and Mosteig and Sweedler~\cite{MS} and for related results.
Especially, non-negative gradings defined by matrices of integers were studied thoroughly by Kreuzer and Robbiano in \cite[Section 4.2]{KR}. They remarked in \cite[p. 15]{KR}: ``For actual computations, arbitrary gradings by matrices are too general". Nevertheless, we can develop an effective theory of leading ideals for monomial preorders despite various obstacles compared to the theory of monomial orders.
\par
 
Let $<$ be an arbitrary monomial preorder of $\Kx$. Following Greuel and Pfister \cite{GP}, we will work in the localization $\Kx_< := S_<^{-1}\Kx$, where $S_< := \{u \in \Kx \mid L_<(u) = 1\}$. Note that $\Kx_< = \Kx$ if and only if $1 < x_i$ or~$1$ and~$x_i$ are incomparable for all~$i$, and $\Kx_< = \Kx_{(X)}$ if and only if $x_i < 1$  for all~$i$. In these cases, we call $<$ a {\em global monomial preorder} or {\em local monomial preorder}, respectively. For every element $f \in \Kx_<$,  we can choose $u \in S_<$ such that $uf \in K[X]$, and define $L_<(f) := L_<(uf)$. The {\em leading ideal} of a set $G \subseteq \Kx_<$ is the ideal in $\Kx$ generated by the polynomials $L_<(f)$, $f \in G$, denoted by $L_<(G)$. \par

Let $I$ be an ideal in $\Kx_<$. For monomial orders, there is a division algorithm and a notion of s-polynomials, which are used to devise an algorithm for the computation of a standard basis of $I$, i.e. a finite set $G$ of elements of $I$ such that $L_<(G) = L_<(I)$. For monomial preorders, there is no such algorithm. However, we can overcome this obstacle by  refining the given monomial preorder $<$ to a monomial order. We shall see that $I$ and $L_<(I)$ share the same leading ideal with respect to such a refinement of the preorder $<$. 
Using this fact, we show that a standard basis of $I$ with respect to the refinement is also a standard basis of $I$ with respect to the original monomial preorder. Therefore, we can compute a standard basis with respect to a monomial preorder by using the standard basis algorithm for monomial orders.  Moreover, we can show that if $J \subseteq I$ are ideals in $\Kx_<$ with $L_<(J) = L_<(I)$, then $J=I$. \par

An important feature of the leading ideal with respect to a monomial order  is that it is a flat deformation of the given ideal \cite{GP}. This can be also shown for a monomial preorder. 
For that we need to approximate a monomial preorder by an integral weight order 
 which yields the same leading ideal. Compared to the case of a monomial order, the approximation for a monomial preorder is more complicated because of the existence of incomparable monomials, which must be given the same weight. \par

Using the approximation by an integral weight order we can relate properties of $I$ and $L_<(I)$ with each other.
The main obstacle here is that $L_<(I)$ and $I$ may have different dimensions. 
However, we always have  $\dim \Kx/L_<(I) = \dim \Kx/I^*$, where $I^* = I \cap \Kx$. 
From this it follows that $\height L_<(I) = \height I$ and $\dim \Kx/L_<(I) \ge \dim \Kx_</I$ with equality if $<$ is a global or local preorder.  Inspired by a conjecture of Kredel and Weispfening \cite{KW} on equidimensionality in Gr\"obner basis theory and its solution by Kalkbrenner and Sturmfels \cite{KS}, we also show that if $\Kx/I^*$ equidimensional, then $\Kx/L_<(I)$ is equidimensional. This has the interesting consequence that if an affine variety is equidimensional at the origin, then so is its tangent cone. \par

Despite the fact that $L_<(I)$ and $I$ may have different dimensions, many properties descend from $L_<(I)$ to $I$. Let $\PP$ be a property which an arbitrary local ring may have or not have.  We denote by $\Spec_\PP(A)$ the $\PP$-locus of  a noetherian ring $A$. If $\PP$ is one of the properties regular, complete intersection, Gorenstein, Cohen-Macaulay, Serre's condition $S_r$, normal, integral, and reduced, we can show that  
$$\dim \Spec_\NP(\Kx_</I) \le \dim \Spec_\NP\bigl(\Kx/L_<(I)\bigr),$$
where $\NP$ denotes the negation of $\PP$.  As far as we know, this inequality is new even for global monomial orders and for the tangent cone.
From this it follows that if $\PP$ holds at all primes of $\Kx/L_<(I)$, then it also holds at all primes of $\Kx_</I$. For a large class of monomial preorders, containing all monomial orders, it suffices to test $\PP$ for the maximal ideal in $\Kx/L_<(I)$ corresponding to the origin.
Moreover, we can show that if $\Kx/L_<(I)$ is an integral domain, then so is $\Kx_</I$. For a positive integral weight order,  Bruns and Conca \cite{BC} showed that the above properties descend from $\Kx/L_<(I)$ to $\Kx/I$. 
However, their method could not be used for monomial preorders. 
 \par

If $I$ is a homogeneous ideal of $\Kx$, we can replace a monomial preorder $<$ by a global monomial preorder, which can be approximated by a positive integral weight order. So we can use results on such weight orders \cite{Cav, Sb, Tr} to compare important graded invariants of $I$ and $L_<(I)$. We can show that the graded Betti numbers of $L_<(I)$ are upper bounds for the graded Betti numbers of $I$. From this it follows that the depth and the Castelnuovo-Mumford regularity of $I$  are bounded by those of $L_<(I)$:   
\begin{align*}
\depth  \Kx /I & \ge \depth \Kx/L_<(I),\\
\reg  \Kx /I & \le \reg \Kx/L_<(I).
\end{align*}
We can also show that the dimension of the graded components of the local cohomology modules of $L_<(I)$ are upper bounds for those of $I$  and that the reduction number of $\Kx/I$ is bounded above by the reduction number of  $\Kx/L_<(I)$. \par

The above results demonstrate that one can use the leading ideal with respect to a monomial preorder to study properties of the given ideal. For some cases, where the preorder is not a total order, the leading ideal still has a  structure like a monomial ideal in a polynomial ring. For instance, if $I$ is an ideal which contains the defining ideal $\Im$ of a toric ring $R$, one can construct a monomial preorder $<$ such that $L_<(I)$ contains $\Im$ and $L_<(I)/\Im$ is isomorphic to a monomial ideal of $R$. This construction was used by Gasharov, Horwitz and Peeva \cite{GHP} to show that if $R$ is a projective toric ring and if $Q$ is an arbitrary homogeneous ideal of $R$, there exists a monomial ideal $Q^*$ in $R$ such that $R/Q$ and $R/Q^*$ have the same Hilbert function. Their result is just a consequence of the general fact that $\Kx/L_<(I)$ and $\Kx/I$ have the same Hilbert function for any homogeneous ideal $I$ and for any monomial preorder $\le$. This case shows that monomial preorders can be used to study subvarieties of a toric variety.
\par

We would like to mention that in a recent paper \cite{KT}, the first two  authors have used global monomial preorders in a polynomial ring over a commutative ring $R$ to characterize the Krull dimension of $R$. 
Global monomial preorders have been also used recently by Sumi, Miyazaki, and Sakata \cite{SMS} to study ideals of minors. 
 \par

The paper is organized as follows. In Section 1 we characterize monomial preorders as products of weight orders, which are given by real matrices. In Section 2 we investigate basic properties of leading ideals.  In Section 3 we approximate a monomial preorder by an integral weight order. Then we use this result to study the dimension  of the leading ideal. In the final Section 4 we prove the descent of properties and invariants from the leading ideal to the given ideal for an arbitrary monomial preorder. 
\par

We refer to the books \cite{Ei} and \cite{GP} for unexplained notions in Commutative Algebra. \par

The authors would like to thank G.-M. Greuel, J. Herzog, J. Majadas, 
G. Pfister, L. Robbiano, T. R\"omer, F.-O. Schreyer, and B. Teissier for stimulating discussions on
the subjects of this paper. We also thank the anonymous referees for their comments.

%%%%%%%%%%%%%%%%%%%%%%%%%%

\section{Monomial preorders}

Recall that a (strict) {\em partial order} on  a set $S$ is a binary relation $<$ on $S$ which is irreflexive, asymmetric, and transitive, i.e., for all $a, b, c \in S$,  
\begin{itemize}
\item not $a < a$;
\item if $a < b$ then not $b < a$;
\item if $a < b$ and $b < c$ then $a < c$. 
\end{itemize}
The elements $a,b$ are said to be {\em comparable} if $a < b$ or $b < a$. One calls $<$ a {\em weak order} if the incomparability is an equivalence relation on $S$. Notice that this is equivalent to saying that the negation~$\not<$ of~$<$ is transitive. A partial order under which every pair of elements is comparable is called a {\em total order}. \par

Let $\Kx = k[x_1,...,x_n]$ be a polynomial ring in~$n$ indeterminates
over a field $k$.  First, we want to see for which (strict) partial
order $<$ on the monomials of $\Kx$ one can define a meaningful
notion of leading polynomials. \par

It is natural that $<$ should be a weak order. Moreover, $<$ should be compatible and cancellative with the multiplication, meaning that $x^a < x^b$ implies $x^a x^c < x^bx^c$ and $x^a x^c < x^bx^c$ implies $x^a < x^b$ for $a,b,c \in \NN^n$. 
We call a weak order $<$ on the monomials of $\Kx$ a {\em monomial preorder} if it the above properties are satisfied. Note that this definition is weaker than the definition of a monomial preorder in \cite{KT}, where it is required that $1 < x^a$ for all $x^a \neq 1$. 
If a monomial preorder is a total order, we call it a {\em monomial
  order}. So a monomial order is precisely what Greuel and
Pfister~\cite[Definition~1.2.1]{GP} call a monomial
ordering. 

\begin{Remark} \label{cancellative}
{\rm For a total order, the cancellative property can be deduced from the compatibility with the multiplication.  
That is no more the case for a weak order. For example, define $x^a < x^b$ if $\deg x^a < \deg x^b$ or $\deg x^a = \deg x^b > 1$ and $x^a <_\lex x^b$.   This weak order is compatible with the product operation but not cancellative because $x_1x_2 < x_1^2$ but $x_2 \not< x_1$.}
\end{Remark}

Monomial preorders are abundant. Given an arbitrary real vector $w \in \RR^n$, we define 
$x^a <_w x^b$ if $w \cdot a < w \cdot b$, with the dot signifying the standard scalar product. 
Obviously, $<_w$ is a monomial preorder. 
One calls $<_w$ the {\em weight order} associated with $w$ \cite{Ei}. For example, the {\em degree order} or the {\em reverse degree order} defined by $x^a < x^b$ if $\deg x^a <  \deg x^b$ or $\deg x^a >  \deg x^b$ is the weight order of  the vector $(1,...,1)$ or $(-1,...,-1)$. More generally, we can associate with every real $m \times n$ matrix $M$ a monomial preorder $<$ by defining $x^a < x^b$ if $M \cdot a <_\lex M\cdot b$, where $<_\lex$ denotes the lexicographic order on $\RR^n$.\par

Given two monomial preorders $<$ and $<'$, we can define a new monomial preorder $<^*$ by $x^a <^* x^b$ if $x^a < x^b$ or if $x^a, x^b$ are incomparable with respect to $<$ and $x^a <' x^b$. We call $<^*$ the {\em product} of $<$ and $<'$. Note that this product is not commutative.
The monomial preorder associated with a real matrix $M$ is just the product of the weight orders associated with the row vectors of $M$. \par 

The following result shows that every monomial preorder of $\Kx$ arises in such a way.

\begin{Theorem} \label{Robbiano}
For every monomial preorder $<$ of $\Kx$, there is a real $m \times n$ matrix $M$ for some $m > 0$ such that  $x^a < x^b$ if and only if $M \cdot a <_\lex M \cdot b$. 
\end{Theorem}

Theorem \ref{Robbiano} is actually about partial orders on $\NN^n$. 
For total orders on $\QQ^n$, it was first shown by Robbiano \cite[Theorem 4]{Ro} (see also \cite[Theorem 2.4]{Ro2}).
For partial orders on $\ZZ^n$, it was shown by Ewald and Ishida \cite[Theorem 2.4]{EI}  from the viewpoint of algebraic geometry. Actually, Ewald and Ishida reduced the proof to the case of total orders on $\QQ^n$. 
However, they were unaware of the much earlier result of Robbiano. We will deduce Theorem \ref{Robbiano} from Robbiano's result by using the following simple observations. These observations also explain why we have to define a monomial preorder as above. Moreover, they will be used later in the course of this paper. \par

Let $S$ be a cancellative abelian monoid with the operation $+$. We call a partial order $<$ on $S$ a {\em partial order of the monoid} $S$ if it is {\em compatible} and {\em cancellative} with $+$, meaning that $a < b$ implies $a + c < b + c$ and $a + c < b + c$ implies $a < b$ for all $a,b,c \in S$. \par

Similarly, if $E$ is a vector space over $\QQ$, a partial order $<$ on $E$ is called a {\em partial order of the vector space} $E$ if it is a partial order of $E$ as a monoid and $a < b$ implies $\lambda a < \lambda b$ for all $\lambda \in \QQ_+$ and $a,b \in E$, where $\QQ_+$ denotes the set of the positive rational numbers.  

\begin{Lemma}\label{extension} 
Every partial order of the additive monoid $\NN^n$ can be uniquely extended to a partial order of the vector space $\QQ^n$.
\end{Lemma}

\begin{proof} 
Let $<$ be a partial order of $\NN^n$.  
For every $a \in \ZZ^n$, there are two unique vectors $a_+, a_- \in \NN^n$ having disjoint supports such that $a = a_+ - a_-$.
For arbitrary $a, b \in \ZZ^n$ we define $a < b$ if $a_+ + b_- < a_- + b_+$. 
One can easily shows that $<$ is a partial order of $\ZZ^n$ extending the partial order $<$ of $\NN^n$.
Now, for arbitrary $a, b \in \QQ^n$, we can  always find a positive integer $p$ such that $pa, pb \in \ZZ^n$. 
We define $a < b$ if $pa < pb$. 
It is easy to see that $<$ is a well-defined partial order of the vector space $\QQ^n$.
\end{proof}

It is clear from the above proof that the cancellative property of $<$ on $\NN^n$ is necessary for the extension of  $<$ to $\QQ^n$. In fact, any partial order on an abelian group which is compatible with the group operation is also cancellative.  \par

If $<$ is a weak order of $\NN^n$, one can easily verify that the extended partial order $<$ on $\QQ^n$ is also a weak order. 

\begin{Lemma}\label{kernel}
Let $<$ be a weak order of the vector space $\QQ^n$. Let $E$ denote the set of the elements which are incomparable to $0$. Then $E$ is a linear subspace of $\QQ^n$ and, if we define $a + E < b+E$ if $a < b$ for arbitrary $a, b \in \QQ^n$, then $<$ is a total order of the vector space $\QQ^n/E$. 
\end{Lemma}

\begin{proof} 
It is clear that two elements $a, b \in \QQ^n$ are incomparable if and only if
$a - b \not < 0$ and $0 \not < a-b$, which means $a - b \in E$. 
Since the incomparability is an equivalence relation, $a, b \in E$ implies $a, b$ are incomparable and, therefore,
 $a-b \in E$. As a consequence, $a \in E$ implies $pa \in E$ for any $p \in \NN$.
From this it follows that $(p/q)a = pa/q \in E$ for any $q \in \ZZ$, $q \neq 0$. 
Therefore, $E$ is a linear subspace of $\QQ^n$ and $a + E$ is the set of the elements which are incomparable to $a$. 
Now, it is easy to see that the induced relation $<$ on $\QQ^n/E$ is a total order of the vector space $\QQ^n/E$.
\end{proof}

Lemma \ref{kernel} does not hold if $<$ is a partial order that is not a weak order. 

\begin{Example}
{\rm Consider the partial order of the vector space $\QQ^n$, $n \ge 2$, defined by the condition $a < b$ if and only if $a - b = \lambda (e_1-e_2)$ for some $\lambda \in \QQ_+$, where $e_i$ denote the standard basis vectors. Then $<$ is not a weak order because $e_1,0$ and $e_2, 0$ are pairs of incomparable elements, whereas $e_1 < e_2$. Clearly, $E$ is not a linear subspace of $\QQ^n$ because $e_1,e_2 \in E$ but $e_1 - e_2 \not\in E$.}
\end{Example}

Now we will use Lemma \ref{extension} and Lemma \ref{kernel} to prove Theorem \ref{Robbiano}. \medskip

\noindent{\em Proof of Theorem \ref{Robbiano}.}
Let $<$ denote the weak order of the additive monoid $\NN^n$ induced by the monomial preorder $<$ in $\Kx$.
By Lemma \ref{extension}, $<$ can be extended to a weak order of $\QQ^n$.
Let $E$ be the set of the incomparable elements to $0$ in $\QQ^n$. 
By Lemma \ref{kernel}, $E$ is a linear subspace of $\QQ^n$ and $<$ induces a total order $<$ of $\QQ^n/E$.
By \cite[Theorem 4]{Ro}, there is an  injective linear map $\phi$ from $\QQ^n/E$ to $\RR^m$ (as a vector space over $\QQ$) such that $a + E < b+E$ if and only if $\phi(a +E) <_\lex \phi(b+E)$ for all  $a, b \in \QQ^n$.
The composition of the natural map from $\QQ^n$ to $\QQ^n/E$ with $\phi$ is
a linear map $\psi$ from $\QQ^n$ to  $\RR^m$ such that
$a < b$ if and only if $\psi(a) <_\lex \psi(b)$.
Since $\psi$ is a linear map, we can find a real $m \times n$ matrix $M$ such that
$\psi(a) = M\cdot a$ for all $a \in \QQ^n$. Therefore, $x^a < x^b$ if and only if $M\cdot a <_\lex M\cdot b$.
\qed \medskip

We shall see in the following remark that a monomial preorder  give rises to a grading
 on $\Kx$, which may be useful for the study of leading ideals.  

\begin{Remark} \label{graded}% 
{\rm Let $<$ be an arbitrary monomial preorder in $\Kx$. 
Let $S$ denote the quotient set of  the monomials with respect to the equivalence relation of
incomparability.  
Since $<$ is compatible and cancellative with  the product of monomials, 
we can define the product of two equivalent classes to make $S$ a totally ordered abelian monoid.  
For  every $a \in \NN^n$ we denote by $[a]$ the equivalent class of the monomials incomparable to $x^a$ and by $\Kx_{[a]}$  the vector space generated by the monomials of $[a]$. 
Then $\Kx = \oplus_{[a] \in S} \Kx_{[a]}$ has the structure of an $S$-graded ring. 
For instance, if $<$ is the weight order associated with a vector $w$, this grading is given by the weighted degree $\deg x^a =  w\cdot a$. We call  a polynomial  or a polynomial ideal {\em $<$-homogeneous} if it is graded with respect to this grading.
It is clear that the leading part of any polynomial is $<$-homogeneous. Therefore, the leading ideal of any set in $\Kx$  is $<$-homogeneous. As a consequence, the leading ideal has a primary decomposition with
$<$-homogeneous primary ideals and $<$-homogeneous associated primes.
See e.g. \cite[Exercise~3.5]{Ei} for more information on rings graded by an abelian monoid
and \cite{Ro2} for algebraic structures over rings graded by a totally ordered abelian group.}
\end{Remark}

We can use the leading ideal of monomial preorders to study different subjects in algebra and geometry.
For instance, if $<$ is the degree order, i.e. $x^a < x^b$ if $\deg x^a < \deg x^b$, 
then $L_<(f)$ is the homogeneous component of the highest degree of a polynomial $f$. 
In this case, the leading ideal $L_<(I)$ of a polynomial ideal $I$ describes the part at infinity of the affine variety $V(I)$
(see e.g. \cite[Definition 4.14]{GP}). If $<$ is the reverse degree order, i.e. $x^a < x^b$ if $\deg x^a > \deg x^b$, 
then $L_<(f)$ is just the homogeneous component of the lowest degree of $f$. In this case, $\Kx/L_<(I)$ is the associated graded ring of $\Kx/I$ with respect to the maximal homogeneous ideal, which corresponds to the concept of the tangent cone (see e.g. \cite[Section 5.4]{Ei}).  \par

In the following we will present a class of useful monomial preorders which arise naturally 
in the study of ideals of toric rings.
Recall that a {\em toric ring} is an algebra $R$ which are generated by a set of monomials 
$t^{c_1},...,t^{c_n}$, $c_1,...,c_n \in \NN^m$, in a polynomial ring $k[t_1,...,t_m]$.
We call an ideal of $R$ a {\em monomial ideal} if it is generated by monomials of $k[t_1,...,t_m]$. 
Monomial ideals of $R$ have a simple structure and can be studied using combinatorics tools. \par 

Let $\phi: \Kx \to R$ denote the map which sends $x_i$ to $t^{c_i}$, $i = 1,...,n$, and $\Im = \ker \phi$.
Then $R = \Kx/\Im$.  One calls $\Im$ the toric ideal of $R$. Every ideal of $R$ corresponds to an ideal of $\Kx$ containing $\Im$. Let $M$ be the matrix of the column vectors $c_1,...,c_n$.
We call  the monomial preorder on $\Kx$ associated to $M$ the {\em toric preorder} associated to $R$.
This order can be used to deform every ideal of $R$ to a monomial ideal.

\begin{Proposition} \label{toric}
Let $R$ be a toric ring and $\Im$ the toric ideal of $R$ in $\Kx$.  
Let $<$ be the toric preorder of $\Kx$ with respect to $R$.
Let $I$ be an arbitrary ideal of $\Kx$ which contains $\Im$.
Then $L_<(I) \supseteq \Im$ and $L_<(I)/\Im$ is isomorphic to a quotient ring of $R$ by a monomial ideal.  
\end{Proposition}

\begin{proof}
It is known that $\Im$ is generated by binomials of the form $x^{a_+} - x^{a_-}$, where $a_+, a_- \in \NN^n$ are two vectors having disjoint supports such that $a = a_+ - a_-$ is a solution of the equation $M \cdot a = 0$ \cite{He}. Since $M \cdot x^{a_+} = M \cdot x^{a_-}$, $x^{a_+}$ and $x^{a_-}$ are incomparable with respect to $<$. Hence, $L_<(x^{a_+} - x^{a_-}) = x^{a_+} - x^{a_-}$. Thus, $L_<(\Im) = \Im$. Since $I \supseteq \Im$, this implies $L_<(I) \supseteq \Im$. \par   

Since $L_<(I)/\Im \cong \phi(L_<(I))$, it remains to show that $\phi(L_<(I))$ is a monomial ideal of $R$. This follows from the general fact that for any polynomial $f \in \Kx$, $\phi(L_<(f))$ is a monomial of $k[t_1,...,t_r]$,  which we shall show below.   \par

If $f$ is a monomial, then $L_<(f) = f$ and $\phi(f)$ is clearly a monomial of $k[t_1,...,t_r]$.
If $f$ is not a monomial, $L_<(f)$ is a linear combination of incomparable monomials. 
Therefore, it suffices to show that if $x^a, x^b$ are two incomparable monomials, then $\phi(x^a) = \phi(x^b)$.
Let $M$ be the matrix defined as above.  Since $<$ is the monomial preorder associated to $M$, 
$M \cdot a = M \cdot b$. Hence,
$\phi(x^a) = t^{M \cdot a} = t^{M \cdot b'} = \phi(x^{b}).$ 
\end{proof}

Proposition \ref{toric} extends a technique used by Gasharov, Horwitz and Peeva to show that if $R$ is a projective toric ring and if $Q$ is a homogeneous ideal in $R$, then there exists a monomial ideal $Q^*$ such that $R/Q$ and $R/Q^*$ have the same Hilbert function \cite[Theorem 2.5(i)]{GHP}.  
 In this case, we have $R/Q \cong k[X]/I$ and $R/Q^* \cong k[X]/L_<(I)$ for some homogeneous ideal $I$.  In the next section we will prove the more general result that if $I$ is an arbitrary homogeneous ideal, then $\Kx/I$ and $\Kx/L_<(I)$ have the same Hilbert function for any homogeneous ideal $I$ of $\Kx$ and any monomial preorder $<$.  

%%%%%%%%%%%%%%%%%%%%%%%%%%%%

\section{Computation of leading ideals}

Let~$<$ be an arbitrary monomial preorder on $\Kx$. 
Since $<$ is compatible with the product operation, we have $L_<(f g) = L_<(f)L_<(g)$ for $f,g \in \Kx$. 
It follows that the set $S_< := \{u \in \Kx \mid L_<(u) = 1\}$ is closed under multiplication, so we can form the localization $\Kx_< := S_<^{-1} \Kx.$ \par

It is easy to see that $S_< = \{1\}$ if and only if $1 < x_i$ or~$1$
and~$x_i$ are incomparable for all~$i$ and that $S_< = \Kx \setminus (X)$ if
and only if $x_i < 1$  for all~$i$. That means $\Kx_< = \Kx$ or $\Kx_< = \Kx_{(X)}$, 
explaining why we call $<$ in these cases a {\em global monomial preorder} or {\em local monomial preorder}.  For monomial orders, these notions coincide with those introduced by Greuel and Pfister  \cite{GP}.    
\par  

For every element $f \in \Kx_<$, there exists $u \in S_<$ such that $uf \in K[X]$.
If there is another $v \in S_<$ such that $vf \in K[X]$, then $L(vf) = L(uvf) = L(uf)$ because $L(u) = L(v) = 1$.
Therefore, we can define $L_<(f) := L_<(uf)$.    
Recall that for a subset $G \subseteq \Kx_<$, the {\em leading ideal} $L_<(G)$ of $G$ is generated by the elements $L_<(f)$, $f \in G$,  in $\Kx$. \par

The above notion of leading ideal allow us to work in both rings $k[X]$ and $\Kx_<$.
Actually, we can move from one ring to the other ring by the following relationship. 
  
\begin{Lemma} \label{leading} 
Let $Q$ be an ideal in $\Kx$ and $I$ an ideal in $\Kx_<$. Then\par
{\rm (a)} $L_<(Q\Kx_<) = L_<(Q),$ \par
{\rm (b)} $L_<(I \cap \Kx) = L_<(I)$.
\end{Lemma}

\begin{proof}
For every $f \in Q\Kx_<$, there exists $u \in S_<$ such that $uf \in Q$. Therefore, $L_<(f) = L_<(uf) \in L_<(Q)$.
This means $L_<(QKx_<) \subseteq L_<(Q)$. Since $Q \subseteq Q\Kx_<$, this implies $L_<(Q \Kx_<) = L_<(Q)$. Now let $Q = I \cap \Kx$. Then $Q\Kx_< = I$. As we have seen above, $L_<(Q) = L_<(I)$.
\end{proof}

By Lemma \ref{leading}(a), two different ideals in $\Kx$ have the same leading ideal if they have the same extensions in $\Kx_<$. This explains why we have to work with ideals in $\Kx_<$. \par

For a monomial order, there is the division algorithm, which gives a
remainder $h$ (or a weak normal form in the language of \cite{GP}) of
the division of an element $f \in \Kx_<$ by the elements of $G$ such that if $h \neq 0$, $L_<(h) \not\in L_<(G)$. 
This algorithm is at the heart of the computations with ideals by monomial
orders \cite{GP}. In general, we do not have a division algorithm for monomial
preorders.  For instance, if $<$ is the monomial preorder without
comparable monomials, then $L_<(f) = f$ for all $f \in \Kx$. In this
case, there are no ways to construct such an algorithm. However, we
can overcome this obstacle by refining the monomial preorder $<$. \par

We say that a monomial preorder $<^*$ in $\Kx$ is a {\em refinement}
of $<$ if $x^a < x^b$ implies $x^a <^* x^b$.  Notice that this implies
$S_< \subseteq S_{<^*}$, so $\Kx_< \subseteq \Kx_{<^*}$. 
The product of $<$ with an other monomial preorder $<'$ is a
refinement of $<$. Conversely, every refinement $<^*$ of $<$ is
the product of $<$ with $<^*$.  

\begin{Lemma} \label{finer}%
  Let $<^*$ be the product of $<$ with a monomial preorder $<'$. Then
  \begin{enumerate}
    \renewcommand{\theenumi}{\alph{enumi}}%
  \item \label{lRefineA} $L_{<^*}(G) \subseteq
    L_{<'}\bigl(L_<(G)\bigr)$ for every subset $G \subseteq \Kx_<$,
  \item \label{lRefineB} $L_{<^*}(I) = L_{<'}\bigl(L_<(I)\bigr)$ for
    every ideal $I \subseteq \Kx_<$,  
  \item \label{lRefineC} if~$<'$ is global, then $\Kx_{<^*} = \Kx_<$.  
  \end{enumerate}
\end{Lemma}

\begin{proof}
  To show part~\eqref{lRefineA}, let $f \in G$ and choose $u \in S_<$ with $u f \in \Kx$.  Then
  \[
  L_{<^*}(f) = L_{<^*}(u f) = L_{<'}\bigl(L_<(u f)\bigr) =
  L_{<'}\bigl(L_<(f)\bigr) \in  L_{<'}\bigl(L_<(G)\bigr).
  \]
 
To show part~\eqref{lRefineB}, we only need to show that $L_{<'}\bigl(L_<(I)\bigr) \subseteq L_{<^*}(I)$.
Let $g \in L_<(I)$.  Then $g = \sum_{i=1}^m h_iL_<(f_i)$ with $h_i \in \Kx$ and $f_i \in I$. 
We may assume that the~$h_i$ are monomials, so $h_i  L_<(f_i) = L_<(h_i f_i)$ for all $i$.
 Replacing the $f_i$ by suitable $u_i f_i$ with $u_i \in S_<$, we may assume $f_i \in I \cap \Kx$. 

Let us first consider the case $g$ is $<$-homogeneous. 
Then we may further assume that  the monomials of all $L_<(h_i f_i)$ are equivalent to the monomials of $g$. Therefore, if we set $f = \sum_{i=1}^m h_if_i$, then $g = L_<(f)$. 
Since $f \in I$, we get 
$$L_{<'}(g) = L_{<'}(L_<(f)) = L_{<^*}(f) \in L_{<^*}(I).$$

Now we drop the assumption that $g$ is $<$-homogeneous. 
Since $L_<(I)$ is $<$-homogeneous, all $<$-homogeneous components of $g$ belong to $L_<(I)$.
As we have seen above, their leading parts with respect to $<'$ belong to $L_{<^*}(I)$.  
Let $g_1,...,g_r$ be those   $<$-homogeneous components of $g$ that contribute terms to
  $L_{<'}(g)$. Since each term of $L_{<'}(g)$ occurs in precisely one $<$-homogeneous component  of $f$, 
 $$L_{<'}(g) = \sum_{j=1}^r L_{<'}(g_j) \in L_{<^*}(I).$$
 Therefore, we can conclude that $L_{<'}(L_<(I)) \subseteq L_{<^*}(I)$. \par
  
To prove part~\eqref{lRefineC} we show that $S_{<^*} = S_<$.
Since $S_< \subseteq  S_{<^*}$, we only need to show that $S_{<^*} \subseteq  S_<$. 
Let $f \in S_{<^*}$. Then $L_{<'}(L_<(f)) = L_{<^*}(f) = 1$.  
Since $<'$ is a global monomial preorder, $1 <' x^a$ or~$1$ and $x^a$ are incomparable for all $x^a \neq 1$.
Therefore, we must have $L_<(f) = 1$, which means $f \in S_<$.  
\end{proof}

The following example shows that the inclusion in Lemma
\ref{finer}\eqref{lRefineA} may be strict.

\begin{Example} {\rm Let $<$ be the monomial preorder without any
    comparable monomials. Then $L_<(f) = f$ for every polynomial $f$.
    Let $<^*$ be the degree reverse lexicographic order. Then $<^*$ is
    the product of $<$ with~$<^*$.  For $G = \{x_1,x_1 + x_2\}$, we
    have \[ L_{<^*}(L_<(G)) = L_{<^*}((x_1,x_1 + x_2)) = (x_1,x_2)
    \supsetneqq (x_1) = L_{<^*}(G).
    \]}
\end{Example}

By Lemma \ref{finer}\eqref{lRefineB}, $I$ and $L_<(I)$ share the same leading ideal with respect to $<^*$. 
If we choose $<'$ to be a monomial order, then $<^*$ is also a monomial order. 
Therefore,  we can use results on the relationship between ideals and their leading ideals in the case of monomial orders to study this relationship in the case of monomial preorders. \par

First, we have the following criterion for the equality of ideals by means of their leading ideals.  

\begin{Theorem} \label{equality}%
  Let $J \subseteq I$ be ideals of $\Kx_<$ such that $L_<(J) = L_<(I)$,
  then $J = I$.
\end{Theorem}

\begin{proof}
  Let $<^*$ be the product of $<$ with a global monomial order $<'$.
  Using Lemma \ref{finer}\eqref{lRefineB}, we have
  \[
  L_{<^*}(J) = L_{<'}\bigl(L_<(J)\bigr) = L_{<'}\bigl(L_<(I)\bigr) =
  L_{<^*}(I).
  \]
 Moreover, $\Kx_< = \Kx_{<^*}$ by Lemma~\ref{finer}\eqref{lRefineC}.
 Since $<^*$ is a monomial order, these facts implies $J = I$  \cite[Lemma~1.6.7(2)]{GP}.
\end{proof}

Let $I$ be an ideal of $\Kx_<$. We call a finite set $G$ of elements
of $I$ a {\em standard basis} of $I$ with respect to $<$ if $L_<(G) = L_<(I)$. 
This means that $L_<(I)$ is generated by the elements $L_<(f)$, $f \in G$.
For monomial orders, our definition coincides with \cite[Definition~1.6.1]{GP}.
If $<$ is a global monomial order, then $\Kx_< = \Kx$ and a standard basis is just a Gr\"obner basis. 

\begin{Corollary} \label{generating}
  Let $G$ be a standard basis of $I$. Then $G$ is a generating set of $I$.
\end{Corollary}

\begin{proof}
  Let $J := (G)$. Then $J \subseteq I$ and $L_<(I) = L_<(G)
  \subseteq L_<(J) \subseteq L_<(I)$. So $L_<(J) = L_<(I)$.
Hence $J = I$ by Theorem  \ref{equality}.
\end{proof}

The above results do not hold for ideals in $\Kx$.
This can be seen from the following observation.  
For every ideal $Q$ of $k[X]$ we define  
$$Q^* := Qk[X]_< \cap k[X].$$ 
Then $Q \subseteq Q^*$. By Lemma \ref{leading},  $L_<(Q) = L_<(Q^*)$.
Therefore, a standard basis of $Q$ is also a standard basis of $Q^*$.
One can easily construct ideals $Q$ such that $Q^* \neq Q$. 
For instance, if $Q = (uf)$ with $1 \neq u \in S_<$ and $0 \neq f \in k[X]$, then $f \in Q^* \setminus Q$.
\par

To compute the leading ideal $L_<(I)$ we only need to compute a standard basis $G$ of $I$  
and then extract the elements $L_<(f)$, $f \in G$, which generate $L_<(I)$. 
The following result shows that the computation of the leading ideal  can be passed to the case of a monomial order. Note that the product of a monomial preorder with a monomial order is always a monomial order.

\begin{Theorem} \label{standard}%
  Let $<^*$ be the product of $<$ with a global monomial order. 
Let $I$ be an ideal in $\Kx_<$ (which by   Lemma~\ref{finer}\eqref{lRefineC} equals $\Kx_{<^*}$). 
Then every standard basis $G$ of $I$ with respect to $<^*$ is also a standard basis of $I$ with respect to $<$.
\end{Theorem}

\begin{proof}
  Let $<^*$ be the product of $<$ with a global monomial order
  $<'$. Let $G$ be a standard basis of $I$ with respect to $<^*$.
By Lemma~\ref{finer}\eqref{lRefineA} and \eqref{lRefineB}, we have
\[
  L_{<'}\bigl(L_<(I)\bigr) = L_{<^*}(I) = L_{<^*}(G) \subseteq
  L_{<'}\bigl(L_<(G)\bigr) \subseteq L_{<'}\bigl(L_<(I)\bigr).
  \]
This implies $L_{<'}(L_<(G)) = L_{<'}(L_<(I))$.  
Therefore, applying Theorem \ref{equality} to $<'$, we obtain $L_<(G) = L_<(I)$.
\end{proof}

If~$<$ is a monomial order, there is an effective algorithm that computes a
standard basis of a given ideal $I \subseteq \Kx_<$ with respect
to~$<$ (see \cite[Algorithm~1.7.8]{GP}).  Since monomial orders are monomial preorders, 
we cannot get a more effective algorithm. For this reason we will not address computational issues like membership test and complexity for monomial preorders. \par

For global monomial preorders defined by matrices of integers, Corollary \ref{generating} and Theorem \ref{standard} were already proved by Kreuzer and Robbiano \cite[Propositions 4.2.14 and 4.2.15]{KR}. Note that they use the term Macaulay basis instead of standard basis. \par

For an ideal $I \subseteq \Kx$, we also speak of a standard basis of
$I$ with respect to a monomial preorder~$<$, meaning a standard basis
$G \subseteq I$ of $I \Kx_<$.

\begin{Theorem}
  Let $I \subseteq \Kx$ be a polynomial ideal. Then the set of all leading ideals of $I$ with respect to
  monomial preorders is finite.
 Hence, there exists a {\em universal standard basis} for $I$, i.e., a finite subset $G \subseteq
  I$ that is a standard basis with respect to all monomial
  preorders.  
\end{Theorem}

\begin{proof} 
 For monomial orders, this result was proved by Mora and Robbiano \cite[Proposition 4.1]{MR}.
 It can be also deduced from a more recent result of Sikora in \cite{Sikora:04} on the compactness of the space of all monomial orders.
 By Theorem \ref{standard}, for each monomial preorder $<$, there exists a monomial order $<^*$ 
 such that every standard basis of $I$ with respect to $<^*$ is also a standard basis of $I$ with respect to $<$.
 Therefore, the set of of all leading ideals of $I$ with respect to monomial preorders is finite.
 \end{proof}
 
In the remainder of this paper, we will investigate the problem whether the leading ideal with respect to a monomial preorder $<$ can be used to study properties of the given ideal. \par

First, we will study the case of homogeneous ideals. 
Here and in what follows, the term ``homogeneous'' alone is
used in the usual sense. In this case we can always replace 
a monomial preorder $<$ by a global monomial preorder. 

\begin{Lemma} \label{homogeneous}%
  Let $I$ be a homogeneous ideal in $\Kx$.  Let $<^*$ be the product
  of the degree order with $<$.  
Then~$1 <^* x_i$ for all $i$ and $L_{<^*}(I) = L_<(I)$.
 \end{Lemma}
 
\begin{proof}
Let $<'$ denote the degree order. Then $1 <' x_i$ for all $i$.
Since $<^*$ is a refinement of $<'$, we also have $1 <^* x_i$ for all $i$.
For every polynomial $f$, $L_{<'}(f)$ is a homogeneous component of $f$. 
In particular, $L_{<'}(f) = f$ if $f$ is homogeneous. Since $I$ is a homogeneous ideal,
every homogeneous component of every polynomial of $I$ belongs to $I$.
Therefore, $L_{<'}(I) = I$. By Lemma \ref{finer}\eqref{lRefineB}, this implies
$L_{<^*}(I) = L_<(L_{<'}(I)) = L_<(I).$
\end{proof}

\begin{Corollary} 
Let $I$ be a homogeneous ideal in $\Kx$. Then $L_<(I)$ is a homogeneous ideal. 
\end{Corollary}

\begin{proof}
By Lemma \ref{homogeneous},  $L_<(I) = L_{<^*}(I)$. 
Since $<^*$ is a refinement of the degree order,  
$L_{<^*}(I)$ is a homogeneous ideal. 
\end{proof}

Let $HP_R(z)$ denote the Hilbert-Poincare series of a standard graded algebra $R$ over $k$, i.e.
$$HP_R(z) := \sum_{t \ge 0} (\dim_k R_t)z^t,$$
where $R_t$ is the vector space of the homogeneous elements of degree $t$ of $R$ and $z$ is a variable. 
Note that $\dim_k R_t$ is the Hilbert function of $R$.

\begin{Theorem} \label{Hilbert}
Let $I$ be a homogeneous ideal in $\Kx$. Then 
$$HP_{\Kx/I}(z)  =  HP_{\Kx/L_<(I)}(z).$$   
\end{Theorem}

\begin{proof}
By 
Let $<^*$ be the product of $<$ with a monomial order $<'$.  
Since $<^*$ is a monomial order, we can apply \cite[Theorem 5.2.6]{GP} to get
$$HP_{\Kx/I}(z)  =  HP_{\Kx/L_{<^*}(I)}(z).$$
Since $L_{<^*}(I) = L_{<'}(L_<(I))$ by Lemma \ref{finer}\eqref{lRefineB}, 
we can also apply  \cite[Theorem 5.2.6]{GP} to $<'$ and obtain
$$HP_{\Kx/L_<(I)}(z)  =  HP_{\Kx/L_{<^*}(I)}(z).$$ 
Comparing the above formulas we obtain the assertion.
\end{proof}

\begin{Corollary} \label{homogen dim}
Let $I$ be a homogeneous ideal in $\Kx$. Then  
$$\dim \Kx/I = \dim \Kx/L_<(I).$$
\end{Corollary}

\begin{proof}
By Theorem \ref{Hilbert}, $\Kx/I$ and $\Kx/L_<(I)$ share the same Hilbert function.
As a consequence, they share the same Hilbert polynomial. 
Since the dimension of a standard graded algebra is the degree of its Hilbert polynomial,
they have the same dimension.
\end{proof}

We shall see in the next section that Corollary \ref{homogen dim} does not hold for arbitrary ideals in $\Kx$ and $\Kx_<$.

%%%%%%%%%%%%%%%%%%%%%%%%%

\section{Approximation by integral weight orders} \label{sWeight}

In the following we call a weight order $<_w$ {\em integral} if $w \in \ZZ^n$.  
The following result shows that on a finite set of monomials, 
any monomial preorder $<$ can be approximated
by an integral weight order.  This result 
is known for monomial orders \cite[Lemma 1.2.11]{GP}.

For a monomial preorders, the approximation may appear to be difficult since we have to dealt with incomparable monomials, which must have the same weight. A complicated proof for global monomial preorders was given by the first  two authors in \cite[Lemma 3.3]{KT}.

\begin{Lemma} \label{approx 1}
For any finite set $S$ of monomials in $\Kx$ we can find $w \in \ZZ^n$ such that $x^a < x^b$ if and only if $x^a <_w  x^b$ for all $x^a, x^b \in S$.
\end{Lemma}

\begin{proof}
Let $<$ denote the weak order of $\NN^n$ induced by the monomial preorder $<$ in $\Kx$.
By Lemma \ref{extension}, $<$ can be extended to a weak order of $\QQ^n$.
By Lemma \ref{kernel}, the set $E$ of the elements incomparable to $0$ is a linear subspace of $\QQ^n$. 
Let $s = \dim \QQ^n/E$. Let $\phi: \QQ^n \to \QQ^s$ be a surjective map such that $\ker \phi = E$. \par
Set $S' = \{\phi(a) - \phi(b)|\ a, b \in S, a < b\}$.
If $\phi(a) - \phi(b) = - (\phi(a') - \phi(b'))$ for $a,b,a',b' \in S$, $a < b$, $a' <b'$, then
$\phi(a+ a') = \phi(b + b')$.
By Lemma \ref{kernel}, this implies that $a+a'$ and $b+b'$ are incomparable, which is a contradiction to the fact that $a + a' < b+b'$. Thus, if $c \in S'$, then $-c \not\in S'$.\par
Now, we can find an integral vector $v \in \ZZ^s$ such that $v\cdot c < 0$ for all $c \in S'$. Thus, $a < b$ if and only if $v \cdot \phi(a) < v \cdot \phi(b)$ for all $a, b \in S$.
We can extend $v$ to an integral vector $w \in \ZZ^n$ such that 
$w \cdot a = w' \cdot \phi(a)$ for all $a \in \QQ^n$. 
From this it follows that $a < b$ if and only if $w \cdot a < w \cdot b$ for all $a, b \in S$.
Hence  $x^a < x^b$ if and only if $x^a <_w  x^b$.
\end{proof}

Using Lemma \ref{approx 1} we can show that  on a finite set of ideals, any monomial preorder $<$ can be replaced by an integral order. The case of several ideals will be needed in the sequel. 

\begin{Theorem} \label{approx 2}%
  Let $I_1,\ldots,I_r$ be   ideals in $\Kx$.  
Then there exists an integral vector $w = (w_1,...,w_n) \in \ZZ^n$ 
such that $L_<(I_i) = L_{<_w}(I_i)$ for $i = 1,...,r$. 
 \end{Theorem}

\begin{proof} 
  Let $<^*$ be the product of $<$ with a global monomial order $<'$. 
Then $\Kx_{<^*} = \Kx_<$ by  Lemma~\ref{finer}\eqref{lRefineC}.
  For each~$i$, let $G_i \subset \Kx$ be a standard basis of $I_ik[X]_<$ with respect to $<^*$.  
Then $L_<(I_i) = L_<(I_i\Kx_<) = L_<(G_i)$ by Lemma \ref{leading}(a) and Theorem~\ref{standard}.  
Since $<^*$ is a monomial order, there exists a finite set $S_i$ of monomials such
  that $G_i$ is a standard basis of $I_i$ with respect to any monomial
  order coinciding with $<^*$ on $S_i$ \cite[Corollary 1.7.9]{GP}. \par
  
Let $S$ be the union of the set of all monomials of the polynomials in the
  $G_i$ with $\cup_{i=1}^rS_i$.  By Lemma \ref{approx 1}, there is an
  integral vector $w \in \ZZ^n$ such that $L_{<_w}(f) = L_<(f)$ for
  all $f \in S$.  This implies $L_<(G_i) = L_{<_w}(G_i)$   for $i = 1,...,r$.  
Let $<_w^*$ be the product of $<_w$ with $<'$.  For all $f \in S$,
  it follows from the definition of the product of monomial orders
  that
  \[
  L_{<_w^*}(f) = L_{<'}(L_{<_w}(f)) = L_{<'}(L_<(f)) = L_{<^*}(f).
  \]
  So $<_w^*$ coincides with $<^*$ on $S_i$.  Therefore, every $G_i$ is a
  standard basis of $I_i$ with respect to $<_w^*$.  By
  Theorem~\ref{standard}, this implies $L_{<_w}(G_i) = L_{<_w}(I_i)$.
  Summing up we get $L_<(I_i) = L_<(G_i) = L_{<_w}(G_i) =
  L_{<_w}(I_i)$.
\end{proof}

Working with an integral weight order has the advantage that  
we can link an ideal to its leading ideal via the homogenization with respect to the weighted degree. \par 

Let $w$ be an  arbitrary vector in $\ZZ^n$. For every polynomial $f = \sum c_ax^a \in k[X]$ we set $\deg_wf  := \max\{w\cdot a|\ c_\a \neq 0\}$ and define
$$f^{\hom}:= t^{\deg_wf}\big(t^{-w_1}x_1,...,t^{-w_n}x_n\big),$$
where $t$ is a new indeterminate and $w_1,...,w_n$ are the components of $w$. 
Then $f^{\hom}$ is a weighted homogeneous polynomial in
$R := k[X,t]$ with respect to the weighted degree $\deg x_i = w_i$ and $\deg t = 1$.
We may view $f^{\hom}$ as the {\em homogenization} of $f$ with respect to $w$ (see e.g. Kreuzer and Robbiano \cite[Section 4.3]{KR}). 
If we write $f^{\hom}$ as a polynomial in $t$, then $L_{<_w}(f)$ is just the constant coefficient of $f^{\hom}$.  \par

For an ideal $I$ in $k[X]$, we denote by $I^{\hom}$ the ideal in $k[X,t]$ generated by the elements $f^{\hom}$, $f \in I$. 
We call $I^{\hom}$ the {\em homogenization of $I$} with respect to $w$. 
Note that $t$ is a non-zerodivisor in $R/I^{\hom}$  \cite[Proposition~4.3.5(e)]{KR}. 
It is clear that
$$L_{<_w}(I)  = (I^{\hom},t)/(t).$$
On the other hand, the map $x_i \to t^{-w_i}x_i$, $i = 1,...,n$, induces an automorphism of $R[t^{-1}]$.
Let $\Phi_w$ denote this automorphism. Then $\Phi_w(f) = t^{-\deg_w}f^{\hom}$. Therefore,
$$\Phi_w(IR[t^{-1}]) = I^{\hom}R[t^{-1}].$$
From these observations we immediately obtain the following isomorphisms. 

\begin{Lemma} \label{iso} 
With the above notations we have \par
{\rm (a)} $R/(I^{\hom},t) \cong k[X]/L_{<_w}(I),$   \par
{\rm (b)} $(R/I^{\hom})[t^{-1}] \cong (k[X]/I)[t,t^{-1}].$
\end{Lemma} 

The above isomorphisms together with the following result show that there is a flat family of ideals over $k[t]$ whose fiber over $0$ is $k[X]/L_{<_w}(I)$ and whose fiber over $t-\lambda$ is $k[X]/I$ for all $\lambda \in k \setminus 0$.
 
\begin{Proposition}  \label{flat}
$R/I^{\hom}$ is a flat extension of $k[t]$.
\end{Proposition}

This result was already stated for an arbitrary integral order $<_w$ by Eisenbud \cite[Theorem~15.17]{Ei}. 
However, the proof there required that all $w_i$ are positive. This case was also proved by Kreuzer and Robbiano in \cite[Theorem 4.3.22]{KR}.
For the case that $w_i \neq 0$ for all $i$, it was proved by Greuel and Pfister \cite[Exercise~7.3.19 and Theorem~7.5.1]{GP}.

\begin{proof} 
   It is known that a module over a principal ideal domain is flat if and only if it is
  torsion-free (see Eisenbud~\cite[Corollary~6.3]{Ei}). Therefore, we only need to
  show that $k[X,t]/I^{\hom}$ is torsion-free.
Let $g  \in k[t]  \setminus \{0\}$ and $F \in k[X,t] \setminus I^{\hom}$.
Then we have to show that $g F   \notin I^{\hom}$.  
Assume that $gF \in I^{\hom}$.
Since $I^{\hom}$ is weighted homogeneous, we may assume that
$g$ and $F$ are weighted homogeneous polynomials. 
Then $g = \lambda t^d$ for some $\lambda \in k$, $\lambda \neq 0$, and $d \ge 0$. 
Since $t$ is a non-zerodivisor in $R/I^{\hom}$, the assumption $gF \in I^{\hom}$ implies $F  \in I^{\hom}$, a contradiction.  
 \end{proof} 

Now we will use the above construction to study the relationship between the dimension of $I$ and $L_<(I)$.
We will first investigate the case $I$ is a prime ideal.  

\begin{Lemma} \label{prime}   
Let $P$ be a prime ideal of $k[X]$ such that $L_<(P) \neq \Kx$. 
Let $Q$ be an arbitrary  minimal prime of $L_<(P)$. Then
$$\dim \Kx/Q = \dim \Kx/P. $$
\end{Lemma}

\begin{proof}   
By Theorem \ref{approx 2} we may assume that $<$ is an integral weight order $<_w$. 
Let $P^{\hom}$ denote the homogenization of $P$ with respect to $w$.
Then $P^{\hom}$ is a prime ideal \cite[Proposition~4.3.10(d)]{KR}. 
By Lemma \ref{iso}(a), there is a minimal prime $Q'$ of $(P^{\hom},t)$ such that $Q \cong Q'/(t)$. 
Since $t$ is a non-zerodivisor in $R/P^{\hom}$, $\height Q' = \height P^{\hom}+1$ by Krull's principal theorem.
By the automorphism $\Phi_w$, 
$\height P^{\hom} = \height P^{\hom}R[t^{-1}] = \height PR[t^{-1}] = \height P$.
Therefore, 
$$\height Q = \height Q' - 1 =  \height P^{\hom} = \height P.$$
Hence, $\dim \Kx/Q = n - \height Q = n- \height P = \dim \Kx/P.$
\end{proof}

It was conjectured by and Kredel and Weispfening \cite{KW} that if $<$ is a global monomial order, then $\Kx/L_<(P)$ is equidimensional, i.e. $\dim \Kx/Q = \dim \Kx/L_<(P)$ for every minimal prime $Q$ of $L_<(P)$. This conjecture was settled by Kalkbrenner and Sturmfels \cite[Theorem 1]{KS} if $k$ is an algebraically closed field (see also \cite[Theorem 6.7]{HT}). Lemma \ref{prime} extends their result to any monomial preorder.   

\begin{Theorem} \label{dim}  
Let $I$ be an ideal of $k[X]$ and $I^* := Ik[X]_< \cap k[X]$. Then \par
{\rm (a)} $\dim k[X]/L_<(I) = \dim k[X]/I^* \le \dim \Kx/I$. \par
{\rm (b)} If $\Kx/I^*$ is equidimensional, then so is $k[X]/L_<(I)$.
\end{Theorem}

\begin{proof}
It is clear that $I^* = \Kx$ if and only if $1 \in I\Kx_< $ if and only if $L_<(I) = \Kx$. 
Therefore, we may assume that $I^* \neq \Kx$. \par

Let $P$ be a minimal prime of $I^*$. Then $P \cap S_< = \emptyset$ because
$P$ is the contraction of a minimal prime of $I\Kx_<$. This means $L_<(P) \neq \Kx$. 
By Proposition \ref{prime}, $\dim \Kx/L_<(P) = \dim \Kx/P.$ 
Choose $P$ such that $\dim \Kx/P = \dim \Kx/I^*$.  
Since $L_<(I) \subseteq L_<(P)$, we have 
$$\dim \Kx/L_<(I) \ge \dim \Kx/L_<(P)  = \dim \Kx/I^*.$$ \par

To prove the converse inequality we use Theorem \ref{approx 2} to choose 
an integral weight order $<_w$ such that $L_<(I) = L_{<_w}(I)$ and $L_<(P) = L_{<_w}(P)$ for all minimal primes $P$ of $I$. Then $L_<(I) \cong (I^{\hom},t)$ and $L_<(P) \cong (P^{\hom},t)/(t)$. \par

Let $Q$ be an arbitrary minimal prime of $L_<(I)$.
Then there is a minimal prime $Q'$ of $(I^{\hom},t)$ such that $Q \cong Q'/(t)$.
Let $P'$ be a minimal prime of $I^{\hom}$ contained in $Q'$. 
Then $Q'$ is also a minimal prime of $(P',t)$.  
By \cite[Proposition 4.3.10]{KR},  
$P' = P^{\hom}$ for some minimal prime $P$ of $I$. 
Hence, $L_<(P) \cong (P',t)/(t).$ 
Therefore, $Q$ is a minimal prime of $L_<(P)$.
By Lemma~\ref{prime}, 
$$\dim \Kx/Q = \dim \Kx/P.$$
Since $(P',t) \subseteq Q'$, $L_<(P) \subseteq Q \neq \Kx$.  
This implies $P \cap S_< = \emptyset$. 
Hence, $P$ is a minimal prime of $I^*$. Therefore, 
$$\dim \Kx/P \le \dim \Kx/I^*.$$
Since there exits $Q$ such that $\dim \Kx/Q = \dim \Kx/L_<(I)$,  we obtain  
$$\dim \Kx/L_<(I) \le \dim \Kx/I^*.$$  
So we can conlude that $\dim k[X]/L_<(I) = \dim k[X]/I^* \le \dim \Kx/I.$ \par

If $\Kx/I^*$ is equidimensional, $\dim \Kx/P = \dim \Kx/I^*$ for all minimal primes $P$ of $I^*$.
As we have seen above, for every minimal prime $Q$ of $L_<(I)$,  there is a minimal prime $P$ of $I^*$ such that $\dim \Kx/Q = \dim \Kx/P$. Therefore, $\dim \Kx/Q = \dim \Kx/I^*$. From this it follows that $\Kx/L_<(I)$ is equidimensional. 
\end{proof}

\begin{Corollary} \label{global}
Let $I$ be an ideal of $k[X]$. Let $<$ be a global monomial preorder. Then\par 
{\rm (a)} $\dim \Kx/L_<(I) = \dim \Kx/I$. \par
{\rm (b)} If $\Kx/I$ is equidimensional, then so is $\Kx/L_<(I)$.
\end{Corollary}

\begin{proof}
For a global monomial preorder $<$, we have $I^* = I$ because $\Kx_< = \Kx$.
Therefore, the statements follow from Theorem \ref{dim}.  
\end{proof}

\begin{Remark}
{\rm If $n \ge 2$ and $<$ is not a global monomial preorder, we can always find an ideal $I$ of $\Kx$ such that
$$\dim \Kx/L_<(I) < \dim \Kx/I.$$
To see this choose a variable $x_i < 1$. Let $I = (x_i-1) \cap (X)$. Then $I^* = (X)$.
By Theorem \ref{dim}(a), $\dim \Kx/L_<(I) = \dim \Kx/I^* = 0$, whereas $\dim \Kx/I = n-1 > 0$.}
\end{Remark}

Now we turn our attention to ideals in the ring $\Kx_<$.  First, we
observe that \linebreak $\dim \Kx_< = n$ because $X$ generates a
maximal ideal of $\Kx_<$ which has height $n$.  However, other maximal
ideals of $\Kx_<$ may have height less than $n$.  The following result
shows that these primes are closely related to the set
$$X_- := \{x_i \mid x_i <1\}.$$

\begin{Lemma} \label{maximal}   
Let $Q$ be a maximal ideal of $k[X]_<$. Then $\height Q = n$ if and only if $X_- \subseteq Q$. 
\end{Lemma}

\begin{proof}    
Assume that $\height Q = n$. Let $Q' = Q \cap \Kx$. Then $\height Q' = \height Q = n$.
Hence $Q'$ is a maximal ideal of $\Kx$.
This implies $Q' \cap k[x_i] \neq 0$ for all $i$. 
Since $Q' \cap k[x_i]$ is a prime ideal, there is a monic irreducible polynomial $f_i$ generating $Q' \cap k[x_i]$. 
For $x_i < 1$,  we must have $f = x_i$ because otherwise $L_<(f_i)$ is the constant coefficient of $f$, which would implies $Q' \cap S_< \neq \emptyset$, a contradiction. Therefore, $X_- \subseteq Q' \subseteq Q$.\par

Conversely, assume that $X_- \subseteq Q$. Then $Q/(X_-)$ is a maximal ideal of the ring $\Kx_</(X_-)$,
which is isomorphic to the polynomial ring $A := k[X \setminus X_-]$ because $A \cap S_< = \emptyset$. 
Therefore, $\height Q/(X_-) =  \dim A = n - \height (X_-)$. Hence 
$$\height Q = \height Q/(X_-) + \height (X_-) = n.$$
\end{proof}

\begin{Theorem} \label{height}%
  Let $I$ be an ideal of $k[X]_<$. Then
  \begin{enumerate}
    \renewcommand{\theenumi}{\alph{enumi}}%
  \item $\height L_<(I) = \height I$,
  \item $\dim \Kx/L_<(I) \ge \dim \Kx_</I$,
  \item $\dim \Kx/L_<(I) = \dim \Kx_</I$ if and only if
    $1 \not\in (P,X_-)$ for at least one prime $P$ of $I$ with
    $\height P = \height I$.
  \end{enumerate}
\end{Theorem}

\begin{proof} 
Let $J = I \cap \Kx$. By Lemma \ref{leading}(b), $L_<(I) = L_<(J)$.  
Since $I = Jk[X]_<$, we have $J^* = J$. 
By Theorem \ref{dim}(a), this implies $\dim k[X]/L_<(J) = \dim k[X]/J$. 
Hence $\height L_<(J) = \height J$.  
By the correspondence between ideals in a localization and their contractions, $\height J = \height I$.
So we can conclude that $\height L_<(I) = \height I$. \par

From this it follows that
$$\dim k[X]/L_<(I) = n - \height L_<(I) = \dim \Kx_<  - \height I \ge \dim \Kx_</I.$$

The above formula also shows that $\dim \Kx/L_<(I) = \dim \Kx_</I$ if and only if  $n  - \height I = \dim \Kx_</I.$ Being a localization of $\Kx$, $\Kx_<$ is a catenary ring. Therefore, the latter condition is satisfied if and only there exists a prime $P$ of $I$ with $\height P = \height I$ such that $P$ is contained in a maximal ideal of height $n$. \par

Assume that a prime ideal $P$ is contained in a maximal ideal $Q$ of height $n$. 
Then $X_- \subset Q$ by Lemma \ref{maximal}. Hence, $1 \not\in (P,X_-)$ because $(P,X_-) \subseteq Q$. 
Conversely, assume that $1 \not\in (P,X_-)$. 
Then, any maximal ideal containing $(P,X_-)$ has height $n$ by Lemma \ref{maximal}.
\end{proof}

We would like to point out the phenomenon that if $I$ is an ideal of $\Kx$, then \linebreak $\dim \Kx/L_<(I) \le \dim \Kx/I$
by Theorem \ref{dim}(a), whereas if $I$ is an ideal of $\Kx_<$,  
then $\dim \Kx/L_<(I) \ge \dim \Kx_</I$ by Theorem \ref{height}(b).\par

\begin{Remark}
{\rm It is claimed in \cite[Corollary 7.5.5]{GP} that 
$$\dim k[X]_</I = \dim k[X]/L_<(I)$$
for any monomial order $<$. This is not true. For instance, let $<$ be the weight order on $k[x, y]$ with weight $(1,-1)$, refined, if desired, to a monomial order. Consider the irreducible polynomial $f = x^2y + 1$ and the ideal $I = (f)$ in $k[x, y]_<$. Since $L_<(f) = x^2y$, $I$ is a proper ideal and since f is irreducible, $I$ is a prime ideal.
Since $1 \in (I,y)$, we have $\dim k[x,y]_</I < \dim k[x,y]/L_<(I)$ by Theorem \ref{height}(c). 
Actually, $I$ is a maximal ideal of $k[x,y]_<$ because any strictly
bigger prime $Q$ has height 2 and must therefore contain $y$ by Lemma
\ref{maximal}. This implies $1 \in Q$, a contradiction.}
\end{Remark}

The following result characterizes the monomial preorders for which
the equality in Theorem~\ref{height}(c) always holds.

\begin{Proposition} \label{cDim}%
  The implications
  \[
  (a) \ \Longrightarrow \ (b) \ \Longleftrightarrow \ (c) \
  \Longleftrightarrow \ (d) \ \Longleftrightarrow \ (e) \
  \Longrightarrow (f)
  \]
  hold for the following conditions on the monomial preorder~$<$:
  \begin{enumerate}
    \renewcommand{\theenumi}{\alph{enumi}}%
  \item The monomial preorder~$<$ is global or local.
  \item The monomial preorder can be defined, in the sense of
    Theorem~\ref{Robbiano}, by a
    real matrix $\left(\begin{smallmatrix} A \\
        B \end{smallmatrix}\right)$ composed of an upper part $A$
    whose entries are all nonpositive, and a lower part $B$ whose
    entries are all nonnegative.
  \item If $x_i < 1$ then $t < 1$ for every monomial~$t$ that is
    divisible by~$x_i$.
  \item Every maximal ideal of $\Kx_<$ has height~$n$.
  \item For every ideal $I \subseteq \Kx_<$, the equality $\dim
    \Kx/L_<(I) = \dim \Kx_</I$ holds.
  \item If $I \subseteq \Kx_<$ is an ideal such that $\Kx_</I$ is
    equidimensional, then also \linebreak $\Kx/L_<(I)$ is
    equidimensional.
  \end{enumerate}
\end{Proposition}

\begin{proof}
  It is clear that~(a) implies~(b) and~(b) implies~(c). One can
  deduce~(b) from~(c) by using that in a matrix defining~$<$ one can
  add a multiple of any row to a lower row. Moreover, (c) holds if and only if
  $L_<(1 + g) = 1$ for every $g \in (X_-)$, which is equivalent to the
  condition that for all $g \in (X_-)$, $1 + g$ is not contained in
  any maximal ideal of $\Kx_<$, or, equivalently, that $X_-$ is
  contained in all maximal ideals. By Lemma~\ref{maximal}, this means
  that the condition~(d) holds.

  By Theorem~\ref{height}(c), the condition~(e) holds if and only if
  $1 \notin (P,X_-)$ for all primes $P \in \Spec(\Kx_<)$, which is
  equivalent to $X_- \subseteq Q$ for all maximal ideals $Q \subset
  \Kx_<$. By Lemma~\ref{maximal}, this means that the condition~(d)
  holds.

  We finish the proof by showing that~(d) implies~(f). If~(d) holds,
  then all primes $P \subset \Kx_<$ satisfy $\height P = n - \dim
  \Kx_</I$. So if $I$ is an ideal with $\Kx_</I$ equidimensional, then
  all minimal primes of $I$ have the same height. Therefore the same
  is true for all minimal primes of $J := \Kx \cap I$. So $J$ is
  equidimensional, and since $J = J^*$, Theorem~\ref{dim}(b) tells us
  that $\Kx/L(J)$ is equidimensional. But $L(I) = L(J)$, and we are
  done.
\end{proof}

For a moment let $I$ be the defining ideal of an affine variety $V$.
If $<$ is the degree order,  then $<$ is a global monomial preorder. In this case, $L_<(I)$ describes the part at infinity of $V$. 
If $<$ is the reverse degree order, then $<$ is a local monomial preorder. In this case, $\Kx/L_<(I)$ corresponds to the tangent cone of $V$ at the origin. 
Therefore, the implication (a) $\Longrightarrow$ (f) of Proposition~\ref{cDim} (a)  has the following interesting consequences.  

\begin{Corollary}
  Let $V$ be an affine variety.
  \begin{enumerate}
    \renewcommand{\theenumi}{\alph{enumi}}%
  \item If $V$ is equidimensional, then so is its part at infinity.
  \item If $V$ is equidimensional at the origin, then so is its
    tangent cone.
  \end{enumerate}
\end{Corollary}

In this context, the question of connectedness is also interesting. A
far reaching result was obtained by Varbaro~\cite{Var}, whose
Theorem~2.5, expressed in the language of this paper, says the
following: If $I \subseteq \Kx$ is an ideal such that $\Spec(\Kx/I)$
is connected in dimension $k \ge 0$ (i.e., its dimension is bigger
than~$k$ and removing a closed subset of dimension less than~$k$ does
not disconnect it), then for any global monomial preorder~$<$, also
$\Spec\left(\Kx/L_<(I)\right)$ is connected in dimension~$k$. The
following examples give a negative answer to the question if this
result carries over to general or local monomial preorders. We thank
F.-O. Schreyer for the second example.

\begin{Example} \label{Schreyer}%
  {\rm  (1)    Let~$<$ be the weight order on $k[x_1,x_2]$ given by $w =
      (1,-1)$. For the prime ideal $I \subseteq k[x_1,x_2]_<$
      generated by $(x_1^2 + 1) x_2 + x_1$, the leading ideal is
      $L_<(I) = \bigl(x_1(x_1 x_2 + 1)\bigr)$. By
      Theorem~\ref{height}, $k[x_1,x_2]_</I$ has dimension~$1$, so its
      spectrum is connected in dimension~$0$. But
      $\Spec\bigl(k[x_1,x_2]/L_<(I)\bigr)$ is not connected.\par
      
    \noindent (2) In $k[x_0,\ldots,x_4]$ consider the polynomials%
      \begin{align*}
        f_1 & = x_0 + x_2 x_3 + x_1 x_4 - x_0 x_4 - x_0^2, \\
        f_2 & = x_3 - x_3 x_4 - x_1 x_3 + x_1 x_2 - x_0 x_3 + x_0 x_2, \\
        f_3 & = x_4 - x_3^2 + x_2 x_3 - x_1^2 - x_0 x_4 + x_0 x_1.%
      \end{align*} The tangent cone at the origin is given by the
      ideal $(x_0,x_3,x_4)$ and, as a short computation shows, at the
      point $(1,0,0,0,0)$ is it is given by $(x_0 + x_4,x_1,x_2)$. The
      projetion $\pi$: $\mathbb{A}^5 \to \mathbb{A}^4$ ignoring the
      first coordinate merges these two points, so applying it to the
      variety $X$ given by the $f_i$ will produce a new variety $Y$
      whose tangent cone at the origin is the union of two planes
      meeting at one point. This can be easily verified, at least in
      characteristic~$0$, by using a computer algebra system such as
      MAGMA~\cite{Mag}.
      Being regular at the origin, $X$ is locally
      integral at the origin, and so the same is true of $Y$. So
      replacing $Y$ by its (only) irreducible component passing through
      the origin, we receive a surface that is connected in
      dimension~$1$, but its tangent space at the origin is not.

      We produced this example by starting with the equations for the
      component of $Y$ through the origin, which were provided to us
      by F.-O. Schreyer.
  }
\end{Example}

%%%%%%%%%%%%%%%%%%%%%%%%%%%%

\section{Descent of properties and invariants} \label{sLoci}

Let $<$ be an arbitrary monomial order in $\Kx$. 
In this section, we will again relate properties of an ideal and its leading ideal. 
Our results follow the philosophy that the leading ideal never behaves better than the ideal itself, 
so the passage to the leading ideal is a ``degeneration.'' \par

First, we will concentrate on the loci of local properties.   
Let $\PP$ denote a property which an arbitrary local ring may have or not have.  
For a noetherian ring $A$ we let $\Spec_\PP(A)$ denote the $\PP$-locus of $A$, i.e. the set of the primes $P$ 
such that the local ring $A_P$ satisfies $\PP$. \par

We say that $\PP$ is an {\em open property}  if for any finitely generated algebra $A$ over a field, $\Spec_\PP(A)$ is a Zariski-open subset of $\Spec (A)$, i.e. $\Spec_\NP(A) = V(Q)$ for some ideal $Q$ of $A$, where $\NP$ is the negation of $\PP$ and 
$$V(Q) := \{P \in \Spec(A) \mid Q \subseteq P\}.$$ 
We say that $\PP$ is a {\em faithful property} if for every noetherian local ring $(A,\mm)$,  the following conditions are satisfied: \par
(F1) If $A[t]_{\mm A[t]}$ has $\PP$, where $t$ is an indeterminate, then $A$ has $\PP$. \par
(F2) If $A/tA$ has $\PP$ for some non-zerodivisor $t \in \mm$, then $A$ has $\PP$.

\begin{Proposition} \label{open}
$\PP$ is open and faithful if $\PP$ is one of the following properties:\par
{\rm (a)} regular, \par
{\rm (b)} complete intersection, \par
{\rm (c)} Gorenstein, \par
{\rm (d)} Cohen-Macaulay, \par
{\rm (e)} $S_r$ ($r \ge 1$), \par
{\rm (f)} normal, \par
{\rm (g)} integral (domain),\par
{\rm (h)} reduced. 
\end{Proposition}

\begin{proof}
It is known that any finitely  generated algebra over a field is excellent \cite[Proposition 7.8.3(ii)]{EGA42}.  
If a ring $A$ is excellent, then $\Spec_\PP(A)$ is open when $\PP$ is (a), (d), (e), (f) \cite[Proposition 7.8.3(iv)]{EGA42}, (b), (c) \cite[Corollary 3.3 and Corollary~1.5]{GM}. If $\PP$ is (g) or (h), $\PP$ is obviously open. \par 

The faithfulness of (a)-(d) is more or less straightforward. 
Since the map $A \to A[t]_{\mm A[t]}$ is faithfully flat, we have (F1)
for (e) and $R_{r-1}$ by \cite[Proposition 6.4.1 and Proposition 6.5.3]{EGA42}.
Since a local ring is reduced or normal if it satisfies $S_1$ and $R_0$ or $S_2$ and $R_1$
\cite[Proposition 5.4.5 or Theorem 5.8.6]{EGA42}, this also proves (F1)
for (f) and (h). For (e), (f) and (h) we have (F2) by \cite[Proposition 2.2 and Corollary 2.4]{CN} for the trivial grading. For (g), (F1) is clear and (F2) follows from \cite[Proposition 3.4.5]{EGA42}.
\end{proof}

The following theorem is the main result of this section.

\begin{Theorem} \label{dim locus}%
Let $\PP$ be an open and faithful property. 
Let $I$ be an ideal of $\Kx_<$. Then 
$$\dim \Spec_\NP\big(\Kx_</I\big) \le \dim \Spec_\NP\big(\Kx/L_<(I)\big).$$
\end{Theorem}
 
As we will see, Theorem \ref{dim locus} follows from the following stronger result, 
which relates the  $\NP$-loci of $\Kx_</I$ and $\Kx/L_<(I)$.

\begin{Theorem} \label{locus}%
Let $\PP$ be an open and faithful property. 
Let $I \subseteq J$ be  ideals in $\Kx_<$ such that  
$V(J/I) \subseteq \Spec_\NP\big(\Kx_</I\big)$. Then
$$V\big(L_<(J)/L_<(I)\big) \subseteq \Spec_\NP\big(\Kx/L_<(I)\big).$$
\end{Theorem}

\begin{proof}
Set $I^* = I \cap \Kx$ and $J^* = J \cap \Kx$. Then $I^* \subseteq J^*$.
By Lemma \ref{leading}, $L_<(I) = L_<(I^*)$ and $L_<(J) = L_<(J^*)$.
Let $P$ be an arbitrary minimal prime of $J^*$ and $\wp$ the corresponding minimal prime of $J$.
Then $(\Kx/I^*)_P = (\Kx_</I)_\wp$.
Since $V(J/I) \subseteq \Spec_\NP\big(\Kx/I\big)$, $(\Kx_</I)_\wp$ does not have $\PP$.
Hence,  $(\Kx/I^*)_P$ does not have $\PP$. 
This shows that $V(J^*/I^*) \subseteq \Spec_\NP\big(\Kx/I^*\big)$. \par 

Now, replacing $I$ and $J$ by $I^*$ and $J^*$ we may assume that $I \subseteq J$ are ideals in $\Kx$ 
such that $V(J/I) \subseteq \Spec_\NP\big(\Kx/I\big)$.  
By Theorem \ref{approx 2} we may assume that $<$ is an integral weight order $<_w$.
Suppose that $V\big(L_<(J)/L_<(I)\big) \not\subseteq \Spec_\NP\big(\Kx/L_<(I)\big).$
Then there exists a minimal prime $P$ of $L_<(J)$ such that
$\big(k[X]/L_<(I)\big)_P$ has $\PP$.
Let $R = k[X,t]$ and $I^{\hom}, J^{\hom}$ be the homogenizations of $I,J$ in $R$ with respect to $w$.
By Lemma \ref{iso}, we have
\begin{align*}
R/(I^{\hom},t) & \cong  k[X]/L_<(I),\\
R/(J^{\hom},t) & \cong  k[X]/L_<(J).
\end{align*} 
Therefore, there exists a minimal prime $P'$ of $(J^{\hom},t)$ such that
$$\big(R/(I^{\hom},t)\big)_{P'} \cong \big(k[X]/L_<(I)\big)_P.$$ 
Since $t$ is a non-zerodivisor in $R/I^{\hom}$, using the faithfulness of $\PP$ we can deduce that
$\big(R/I^{\hom}\big)_{P'}$ also has $\PP$. \par

Let $Q'$ be a minimal prime of $J^{\hom}$ such that $Q' \subseteq P'$.
Since $\PP$ is an open property, $\big(R/I^{\hom}\big)_{Q'}$ also has $\PP$.
Since $t$ is a non-zerodivisor in $R/J^{\hom}$, $t \not\in Q'$. Therefore,
$Q'R[t^{-1}]$ is a prime ideal and
$$\big(R/I^{\hom}\big)_{Q'} = (R/I^{\hom})[t^{-1}]_{Q'R[t^{-1}]}.$$
Let $\Phi_w$ be the automorphism of $R[t^{-1}]$ introduced before Lemma \ref{iso}.
We know that $\Phi_w(I^{\hom}R[t^{-1}]) = IR[t^{-1}]$ and $\Phi_w(J^{\hom}R[t^{-1}])= JR[t^{-1}]$.
Thus, $\Phi_w(Q'R[t^{-1}]) = QR[t^{-1}]$ for some minimal prime $Q$ of $J$ and 
$$(R/I^{\hom})[t^{-1}]_{Q'R[t^{-1}]} \cong (R/IR)[t^{-1}]_{QR[t^{-1}]}.$$
It is easy to see that
$$(R/IR)[t^{-1}]_{QR[t^{-1}]} = (k[X]/I)[t]_{QR}.$$
Therefore, $(k[X]/I)[t]_{QR} \cong \big(R/I^{\hom}\big)_{Q'}$ has $\PP$. 
Since $\PP$ is faithful, $k[X]/I$ also has $\PP$. 
So we obtain a contradiction to the assumption that $V(J/I) \subseteq \Spec_\NP\big(\Kx/I\big)$.
\end{proof}

Now, we are ready to prove Theorem~\ref{dim locus}.

\begin{proof}[Proof of Theorem~\ref{dim locus}]
 Let $J$ be the defining ideal of the $\NP$-locus of $\Kx_</I$, i.e., \linebreak $V(J/I) = \Spec_\NP\big(\Kx/I\big)$.
 Then 
$\dim  \Spec_\NP\big(\Kx_</I\big)  = \dim \Kx_</J$.
 By Theorem~\ref{height}(b), $\dim \Kx_</J  \le \dim \Kx/L_<(J)$.
 By Theorem~\ref{locus},  $V\big(L_<(J)/L_<(I)\big) \subseteq \Spec_\NP\big(\Kx/L_<(I)\big).$
 Hence, 
$\dim \Kx/L_<(J) \le \dim \Spec_\NP\big(\Kx/L_<(I)\big).$
 So we can conclude that $\dim \Spec_\NP\big(\Kx_</I\big)  \le \dim  \Spec_\NP\big(\Kx/L_<(I)\big)$.  
\end{proof}

\begin{Remark}
{\rm Theorem \ref{locus} still holds if we replace the assumption on the openess of $\PP$ by the weaker condition that if $A_P$ has $\PP$ then so is $A_Q$ for all primes $Q \subset P$. This condition is actually used in the  proof of Theorem \ref{locus}. The openess of $\PP$ is only needed to have the dimension of the $\PP$-loci in Theorem \ref{dim locus}.
Moreover, one can also replace property (F2) by the weaker but more complicated condition that $A$ has $\PP$  if $A/tA$ has $\PP$ for some non-zerodivisor $t$ of $A$ such that $A$ is flat over $k[t]$, where $A$ is assumed to be a local ring essentially of finite type over $k$. In fact, we have used (F2)  for a local ring which is of this type by Proposition \ref{flat}. This shows that Theorems~\ref{dim locus} and~\ref{locus} extend to the case that $\PP$ is one of the following properties: the Cohen-Macauly defect or the complete intersection defect is at most~$r$, where $r$ is a fixed integer.}
\end{Remark}

The proof of Theorem \ref{locus} shows that it also holds for ideals in $\Kx$.
However, the following example shows that Theorem \ref{dim locus} does not hold if $I$ is an ideal of $\Kx$.

\begin{Example}  
{\rm Consider an affine variety that has the origin as a regular point
  but has singulatities elsewhere, such as the curve given by $I =
  \bigl(y^2 - (x - 1)^2x\bigr) \subseteq k[x,y]$ with char$(k) \ne
  2$. In such an example, if $\PP$ is the property {\em regular}, we
  have $\dim \Spec_\NP\big(\Kx/I\big) \ge 0$ but $\dim
  \Spec_\NP\big(\Kx/L_<(I)\big) < 0$.
% Let $I = (1+x)^2 \cap (x,y^2)$ in $k[x,y]$. Let $<$ be a monomial preorder on $k[x,y]$ with $x < 1 < y$. Let $\PP$ be the reduced property. Then $\Spec_\NP(k[x,y]/I) = V(J/I)$, where $J = (1+x) \cap (x,y)$. 
% Since $L_<(I) = (x,y^2)$, we have $\Spec_\NP(k[x,y]/L_<(I) = V((x,y)/L_<(I))$. Hence
% $$\Spec_\NP\big(k[x,y]/I\big) = 1 > \dim \Spec_\NP\big(k[x,y]/L_<(I)\big) = 0.$$
}
\end{Example}

Theorem \ref{locus} shows that if $\Spec_\NP(\Kx/L_<(I)) = \emptyset$, then $\Spec_\NP(\Kx_<(I)) = \emptyset$. 
Hence, we has the following consequence. 

\begin{Corollary}  \label{all}
Let $\PP$ be an open and faithful property. Let $I$ be  an ideal in $\Kx_<$.
If $\PP$ holds at all primes of $\Kx/L_<(I)$, then it also holds at all primes of  $\Kx_</I$. 
\end{Corollary}

For a positive integral weight order $<_w$, Bruns and Conca \cite[Theorem 3.1]{BC} shows that the properties Gorenstein, Cohen-Macaulay, normal,  integral, reduced are passed from $\Kx/L_{<_w}(I)$ to $\Kx/I$.  Their proof is based on the positively graded structure of $\Kx$ induced by $w$, which is not available for any integral weight order. \par 

The following corollary gives a reason why it is often easier to work with $L_<(I)$ instead of $I$.  

\begin{Corollary} \label{MthenAll}%
  Let $\PP$ be an open and faithful property, and assume that the
  monomial order~$<$ is such that $1$ is comparable to all other
  monomials. (This assumption is satisfied if $x_i > 1$ for all~$i$ or
  if $<$ is local or if $<$ is a monomial order.)  Let $I$ be a proper
  ideal in $\Kx_<$.  If $\PP$ holds at the maximal ideal
  $\mathfrak{m} = (X)/L_<(I)$ of $\Kx/L_<(I)$, then it also holds at
  all primes of $\Kx_</I$.
\end{Corollary}

\begin{proof}
  Assume that $\Spec_\NP(\Kx_</I) \ne \emptyset$. Then the ideal $J$ in
  Theorem~\ref{locus} can be chosen to be proper. Therefore $L_<(J)$
  is also a proper ideal, and from the hypothesis on $<$ and the fact
  that $L_<(J)$ is $<$-homogeneous it follows that $L_<(J) \subseteq
  (X)$. By Theorem~\ref{locus} this implies that $P$ does not hold at
  $\mathfrak{m}$.
\end{proof}

Moreover, we can also prove the descent of primality.  

\begin{Theorem} \label{tIntegral}%
Let $I$ be an ideal of $\Kx_<$  such that $L_<(I)$ is a prime ideal.
Then $I$ is a prime ideal.
\end{Theorem}

\begin{proof} 
  Choose a global monomial order~$<'$ and let $<^*$ be the product of
  $<$ with~$<'$. 
Then $<^*$ is a monomial order, and $\Kx_{<^*} = \Kx_<$ by 
  Lemma \ref{finer}(c). 
 Let $G$ be a  standard basis of $I$ with respect to~$<^*$. 
We have to show that  if $f,g \in \Kx_< \setminus I$, then $f g \not\in I$.
Without restriction we may replace $f, g$ by their weak normal forms 
with respect to $G$ (see \cite[Definition~1.6.5]{GP}). 
Then $L_{<^*}(f) \notin L_{<^*}(I)$ and  $L_{<^*}(g) \notin L_{<^*}(I)$. Using Lemma~\ref{finer} we obtain
  \[
  L_{<'}\bigl(L_<(f)\bigr) = L_{<^*}(f) \notin L_{<^*}(I) =
  L_{<'}\bigl(L_<(I)\bigr),
  \]
  so $L_<(f) \notin L_<(I)$. Similarly, $L_<(g) \notin L_<(I)$.  By our
  hypothesis, this implies $L_<(f g) = L_<(f)L_<(g) \notin L_<(I)$,
  so $f g \not\in  I$ as desired.
\end{proof}

According to our philosophy that the leading ideal with respect to a
monomial preorder is a deformation that is ``closer'' to the original
ideal than the leading ideal with respect to a monomial order, it
would be interesting to see an example where $\Kx/L_<(I)$ is
Cohen-Macaulay but $\Kx/L_{<^*}(I)$ is not. If~$<$ is a monomial
preorder satisfying the hypothesis of the last statement from
Theorem~\ref{locus}, then the benefit arising from this is that the
Cohen-Macaulay property of $\Kx_</I$ can be verified by testing only
the maximal ideal $\mathfrak{m} := (X)/L_<(I)$ of
$\Kx/L_<(I)$. The following is such an example.

\begin{Example} \label{exCM}%
{\em Consider the ideal
    \[
    I = \bigl(x_1^2,x_2^2,x_3^3, x_1 x_2, x_1 x_3, x_1 x_4 - x_2 x_3 +
    x_1\bigr) \subseteq k[x_1,x_2,x_3,x_4].
    \]
    Let~$< = <_\w$ be the weight order with weight $\w = (1,1,1,1)$,
    and let~$<^*$ be the product of~$<$ and the lexicographic order
    with $x_1 < x_2 < x_3 < x_4$. So~$<^*$ is the graded lexicographic
    order, and it is easy to see by forming and reducing s-polynomials
    that the given basis of $I$ is a Gr\"obner basis with respect
    to~$<^*$. So by Theorem~\ref{standard}, $G$ it is also a standard
    basis with respect to~$<$. So
    \[
    L_<(I) = \bigl(x_1^2,x_2^2,x_3^3, x_1 x_2, x_1 x_3, x_1 x_4 - x_2
    x_3\bigr).
    \]
    From the leading ideal $L_{<^*}(I) = \bigl(x_1^2,x_2^2,x_3^3, x_1
    x_2, x_1 x_3, x_1 x_4\bigr)$ we see that the following elements
    form a vector space basis of $A := \Kx/L_<(I)$:
    \[
    \overline{x_4^i},\ \overline{x_2 x_4^i},\ \overline{x_3 x_4^i},\
    \overline{x_2 x_3 x_4^i} \quad (i \ge 0), \quad \text{and} \quad
    \overline{x_1}.
    \]
    Here the bars indicate the class in $A$ of a polynomial. Because
    $\overline{x_2 x_3 x_4^i} = \overline{x_1 x_4^{i+1}}$ this implies
    \[
    A = k[\overline{x_4}] \oplus k[\overline{x_4}] \cdot
    \overline{x_1} \oplus k[\overline{x_4}] \cdot \overline{x_2}
    \oplus k[\overline{x_4}] \cdot \overline{x_3},
    \]
    and $\overline{x_4}$ is transcendental. It follows that $A =
    \Kx/L_<(I)$ is Cohen-Macaulay, and so the same is true for
    $\Kx/I$.

    Now we turn to $A^* := \Kx/L_{<^*}(I)$. A vector space basis of
    $A^*$ is given as above, but now the bars indicate classes in
    $A^*$. So~$\overline{x_4}$ forms a homogeneous system of
    parameters, but it is not regular since $\overline{x_1}
    \overline{x_4} = 0$. Therefore $A^* = \Kx/L_{<^*}(I)$ is not
    Cohen-Macaulay. }
\end{Example}

In the following we will compare graded invariants of homogeneous ideals with those of its leading ideals. 
The following result is essentially due to Caviglia's proof of Sturmfels' conjecture on the Koszul property of the pinched Veronese \cite{Cav}.

\begin{Proposition} \label{Tor}
Let $I,J,Q$ be homogeneous ideals in $\Kx$. Then
$$\dim_k\Tor_i^{\Kx/I}(\Kx/J,\Kx/Q)_j \le \dim_k\Tor_i^{\Kx/L_<(I)}(\Kx/L_<(J),L_<(Q))_j$$
for all $i \in \NN$, $j \in \ZZ$.
\end{Proposition}

\begin{proof}
By Lemma \ref{homogeneous} we may assume that~$<$ is a monomial
preorder with $1 < x_i$ for all $i$.
Applying Theorem \ref{approx 2} to $I, J, Q$  
we can find $w \in \ZZ^n$ with $w_i > 0$ for all $i$
such that $L_<(I) = L_{<_w}(I)$,  $L_<(J) = L_{<_w}(J)$, and $L_<(Q) = L_{<_w}(Q)$.
For a positive weight vector $w$,  Caviglia \cite[Lemma 2.1]{Cav} already showed that 
$$\dim_k\Tor_i^{\Kx/I}(\Kx/J,\Kx/Q)_j \le \dim_k\Tor_i^{\Kx/L_{<_w}(I)}(\Kx/L_{<_w}(J),L_{<_w}(Q))_j$$
for all $i \in \NN$, $j \in \ZZ$.
\end{proof}

Recall that a $k$-algebra $R$ is called Koszul if $k$ has a linear free resolution as an $R$-module or, equivalently,
if $\Tor_i^R(k,k)_j = 0$ for all $j \neq i$.

\begin{Corollary}
Let $I$ be a homogeneous ideal in $\Kx$. 
If $k[X]/L_<(I))$ is a Koszul algebra, then so is $k[X]/I$.
\end{Corollary}

\begin{proof}
We apply Lemma \ref{Tor} to the case $J = Q = (X)$.
From this it follows that if $\Tor_i^{\Kx/L_<(I)}(k,k)_j = 0$, then $\Tor_i^{\Kx/I}(k,k)_j = 0$ for all $j \neq i$.
\end{proof}

For any finitely generated graded $\Kx$-module $E$, let $\beta_{i,j}(E)$ denote the number of copies of the graded free module $\Kx(-j)$ appearing in the $i$-th module of the resolution the largest degree of a minimal graded free resolution of $E$.
These numbers are called the {\em graded Betti numbers} of $E$. 
In some sense, these invariants determine the graded structure of $E$. 
It is well known that $\beta_{i,j}(E) = \dim_k \Tor_i^{\Kx}(E,k)_j$ for all $i \in \NN$, $j \in \ZZ$.

\begin{Proposition} \label{Betti}
Let $I$ be a homogeneous ideal in $\Kx$. 
Then $\beta_{i,j}(\Kx/I) \le \linebreak \beta_{i,j}(\Kx/L_<(I))$ for all $i \in \NN$, $j \in \ZZ$.
\end{Proposition}

\begin{proof}
We apply Lemma \ref{Tor} to the case $I = 0$, $Q = (X)$ and replace $J$ by $I$. Then
$$\dim_k\Tor_i^{\Kx}(\Kx/I,k)_j \le \dim_k\Tor_i^{\Kx}(\Kx/L_<(I),k)_j$$
which implies $\beta_{i,j}(\Kx/I) \le \beta_{i,j}(\Kx/L_<(I))$ for all $i \in \NN$, $j \in \ZZ$.
\end{proof}

Using the graded Betti numbers of $E$ one can describe other important invariants of $E$ such that the depth and the Castelnuovo-Mumford regularity:
\begin{align*}
\depth E & = n- \max\{i|\  \beta_{i,j} \neq 0 \text{ for some } j\},\\
\reg E & =  \max\{j-i|\ \beta_{i,j} \neq 0\}.
\end{align*}
By this definition, we immediately obtain from Proposition \ref{Betti} the following relationship between the depth and the regularity of $\Kx/I$ and $\Kx/L_<(I)$.

\begin{Corollary} \label{depth}
Let $I$ be a homogeneous ideal in $\Kx$. Then
\begin{align*}
\depth(\Kx/I)  & \ge  \depth (\Kx/L_<(I)),\\ 
\reg(\Kx/I)  & \le  \reg (\Kx/L_<(I)).
\end{align*}  
\end{Corollary}

Let $\mm$ denote the maximal homogeneous ideal of $\Kx$.
For any finitely generated graded $\Kx$-module $E$, we denote by $H_\mm^i(E)$  
the $i$-th local cohomology module of $E$ with respect to $\mm$ for all $i \in \NN$.
Note that $H_\mm^i(E)$ is a $\ZZ^n$-graded module.
As usual, we denote by $H_\mm^i(E)_j$ the $j$-th component of $H_\mm^i(E)$ for all $j \in \ZZ$.  
It is  known that the vanishing of $H_\mm^i(E)$ gives important information on the structure of $E$.  

\begin{Proposition} \label{Sbarra}
Let $I$ be a homogeneous ideal in $\Kx$. Then
$$\dim_k H_\mm^i(\Kx/I)_j \le \dim_k H_\mm^i(\Kx/L_<(I))_j$$ 
for all $i \in \NN, j \in \ZZ$.
\end{Proposition}

\begin{proof}
Sbarra \cite[Theorem 2.4]{Sb} already proved the above inequality for an arbitrary global monomial order. 
Actually, his proof shows that for an arbitrary integral vector $<_w$, 
$$\dim_k H_\mm^i(\Kx/I)_j \le \dim_k H_\mm^i(\Kx/L_{<_w}(I))_j$$ 
for all $i \in \NN, j \in \ZZ$. 
By Theorem \ref{approx 2}, there exists $w \in \ZZ^n$ such that $L_<(I) = L_{<_w}(I)$. 
Therefore, Sbarra's result implies the conclusion.
\end{proof}

Let $R$ be a standard graded algebra over an infinite field $k$ with $d = \dim R$.  
An ideal $Q$ of $R$ is called a {\em minimal reduction} of $R$  if $Q$ is generated by 
a system of linear forms $z_1,\ldots,z_d$ such that 
$k[z_1,\ldots,z_d] \hookrightarrow R$ is a Noether normalization.
Let $r_Q(R)$ denote the maximum degree of the generators of $R$ as a graded $k[z_1,\ldots,z_d]$-module.  One calls  the invariant
$$ r(R) := \min\{r_Q(R)|\ \text{$Q$ is a minimal reduction of $R$}\}$$
the {\it reduction number} of $R$ \cite{Va}. 
\par

The following result on the reduction number of the leading ideal was a conjecture
of Vasconcelos for global monomial orders \cite[Conjecture 7.2]{Va}.  
This conjecture has been confirmed independently by Conca \cite[Theorem 1.1]{Co} 
and the second author \cite[Corollary 3.4]{Tr}. Now we can prove it for monomial preorders.  

\begin{Proposition} \label{reduction}
Let $I$ be an arbitrary homogeneous ideal in $\Kx$. Then
$$r(\Kx/I)  \le r(\Kx/L_<(I)).$$  
\end{Proposition}

\begin{proof}
By Theorem \ref{approx 2}, there exists $w \in \ZZ^n$ such that $L_<(I) = L_{<_w}(I)$.
By \cite[Theorem 3.3]{Tr}, we know that
$r(\Kx/I)  \le r(\Kx/L_{<_w}(I))$
for an arbitrary weight order $<_w$. 
\end{proof}

%%%%%%%%%%%%%%%%%%%%%%%%%%%%

\end{document}